\documentclass[11pt,reqno,a4paper]{amsart}
\usepackage[utf8]{inputenc}
\usepackage[T1]{fontenc}
\usepackage{amssymb}
\usepackage{pgf,tikz,pgfplots}
\pgfplotsset{compat=1.12}
\usepackage{mathrsfs}
\usetikzlibrary{arrows}
\usepackage{bm}
\usepackage[alpine]{ifsym}
\usepackage{xpicture}
\usepackage{calculator}
\usepackage{graphicx}
\usepackage{enumerate}
\usepackage{enumitem}
\usepackage{xspace}

\usepackage[a4paper,top=3.3cm,bottom=3cm,left=3cm,right=3cm,bindingoffset=5mm]{geometry}

\usepackage{color}

\usepackage[colorlinks=true]{hyperref}
\newtheorem{theorem}{Theorem}[section]
\newtheorem{lemma}[theorem]{Lemma}
\newtheorem{corollary}[theorem]{Corollary}
\newtheorem{proposition}[theorem]{Proposition}

\theoremstyle{definition}
\newtheorem{example}[theorem]{Example}
\newtheorem{remark}[theorem]{Remark}
\newtheorem{definition}[theorem]{Definition}
\newtheorem{algorithm}[theorem]{Algorithm}

\numberwithin{equation}{section}

\newcommand{\Z}{\mathbb{Z}}
\newcommand{\Q}{\mathbb{Q}}

\newcommand{\C}{\mathbb{C}}

\newcommand{\Sing}{{\rm Sing}}

\def\fd{\mathbb{F}_\delta}

\def\fol{$\mathcal{F}$\xspace}

\def\zf{Z_\mathcal{F}}

\usepackage[normalem]{ulem}

\title[Algebraic integrability by extension to Hirzebruch surfaces]{Algebraic integrability of planar polynomial vector fields by extension to Hirzebruch surfaces}
\author[C. Galindo]{Carlos Galindo}
\address{Universitat Jaume I, Campus de Riu Sec, Departamento de Matem\'aticas \& Institut Universitari de Matem\`atiques i Aplicacions de Castell\'o, 12071
Caste\-ll\'on de la Plana, Spain.}\email{galindo@uji.es}

\author[F. Monserrat]{Francisco Monserrat}
\address{Instituto Universitario de
Matem\'atica Pura y Aplicada, Universidad Polit\'ecnica de Valencia,
Camino de Vera s/n, 46022 Valencia (Spain).}
\email{framonde@mat.upv.es}

\author[E. P\'erez-Callejo]{Elvira P\'erez-Callejo}
\address{Universitat Jaume I, Campus de Riu Sec, Departamento de Matem\'aticas \& Institut Universitari de Matem\`atiques i Aplicacions de Castell\'o, 12071
Caste\-ll\'on de la Plana, Spain.} \email{callejo@uji.es}

\thanks{Partially supported by MCIN/AEI/10.13039/501100011033 and by ``ERDF A way of making Europe", grants  PGC2018-096446-B-C22 and RED2018-102583-T, by MCIN/AEI/10.13039/501100011033 and by "ESF Investing in your future”, grant PRE2019-089907, as well as by Universitat Jaume I, grant UJI-B2021-02.}

\begin{document}
\begin{abstract}
We study algebraic integrability of complex planar polynomial vector fields $X=A (x,y)(\partial/\partial x) + B(x,y) (\partial/\partial y) $ through extensions to Hirzebruch surfaces. Using these extensions, each vector field $X$ determines two infinite families of planar vector fields that depend on a natural parameter which, when $X$ has a rational first integral, satisfy strong properties about the dicriticity of the points at the line $x=0$ and of the origin. As a consequence, we obtain new necessary conditions for algebraic integrability of planar vector fields and, if $X$ has a rational first integral, we provide a region in $\mathbb{R}_{\geq 0}^2$ that contains all the pairs $(i,j)$ corresponding to monomials $x^i y^j$ involved in the generic invariant curve of $X$.
\end{abstract}

\maketitle
\section{Introduction}
The study of algebraic solutions of ordinary differential equations goes back to the 19th century \cite{Sch73-per} and the problem of deciding whether all their solutions are algebraic was completely solved for linear homogeneous differential equations of the second order with rational coefficients in the mid 20th century (see \cite{kle56-per} for example). The next step was to consider polynomial differential equations of first order and first degree. Remarkable mathematicians as Darboux \cite{dar}, Poincar\'e \cite{poi11, poi12, poi13, poi21}, Painlev\'e \cite{pai} and Autonne \cite{aut} were authors of seminal papers in this topic. One of the main proposed problems was to characterize those complex planar differential systems which are algebraically integrable. As said, this problem is more than a century old and, despite the fact that much progress has been made, it remains unsolved. The question of deciding about algebraic integrability is related to other very interesting problems as to bound the number of limit cycles of a (real) differential system \cite{lli2, lli,Llibresym} or the center problem \cite{sch, d-l-a}.

When a complex planar polynomial differential system is algebraically integrable, with primitive rational first integral $f/g$, $f, g \in \mathbb{C}[x,y]$ (see Definition \ref{primitive}), all  their invariant curves are algebraic and their irreducible components are those of the elements of the pencil of curves with equations $\alpha f + \beta g=0$, where $(\alpha:\beta)$ runs over the complex projective line \cite{zam1}. Only finitely many pairs $(\alpha:\beta)$ (named remarkable values) give rise to reducible polynomials $\alpha f + \beta g$ in $\mathbb{C}[x,y]$. Remarkable values are of importance for the phase portrait of the system \cite{CGGL2, FL}. Jouanolou in \cite{Joua} (see also \cite{CL2} and \cite{Ch,CLS,Che}) proved that the existence of enough algebraic invariant curves for the planar differential system implies its algebraic integrability.

Poincaré \cite{poi11, poi12, poi13, poi21}, considering an algebraically integrable planar differential system and looking for a rational first integral, posed the problem of giving an upper bound on the degree of the first integral which depends only on the degree of the differential system. Lins-Neto in \cite{l-n} proved that there is no such a bound even for families of systems where the analytic type and the number of singularities of the associated vector field remain constant. This problem has attracted a lot of attention and one can find many articles related to it (even in higher dimension) \cite{poi21, ce-li, car,ca-ca,zam1,soa1,soa2,Walcher,zam2,pere,es-kl,c-l,GalMon2010,GalMon2014,PerSva}. The literature also contains several algorithms which, in specific cases, allow us to compute first integrals \cite{GalMon2006,fergia,GalMon2014, fergal1, bost, fergal2}.

We are concerned with  rational first integrals of planar systems; however other interesting classes of first integrals  (like elementary, Liouvillian or Darboux first integrals) are being studied (see for instance the survey \cite{LZDarb} and references therein).

Many of the recent advances on algebraic integrability and the Poincar\'e problem have been obtained by considering foliations of the projective plane as one can see in several previous references. Since the involved foliations have singularities, a useful tool for addressing integrability and related problems consists of successively blowing up the projective plane and the blown up surfaces at the ordinary singularities of the foliation to reach a transformed foliation with at most simple singularities. It is well-known that the relatively minimal models of smooth complex rational surfaces $Z$ are the complex projective plane $\mathbb{P}^2$ and the complex Hirzebruch surfaces $\mathbb{F}_\delta$, $\delta \neq 1$ being a nonnegative integer. This means that $Z$ can be obtained from $\mathbb{P}^2$ or $\mathbb{F}_\delta$ after finitely many blowups.

In this paper, instead of considering foliations on $\mathbb{P}^2$, we propose to extend complex planar polynomial vector fields $X$ to foliations on Hirzebruch surfaces $\mathbb{F}_\delta$, $\delta \geq 0$. In this way, we have an extension $\mathcal{F}_X^\delta $ for each nonnegative integer $\delta$ corresponding to a complex Hirzebruch surface. We hope that this procedure will contribute to substantial progress towards the above described problems.

Section \ref{prelim} briefly recalls the essentials about algebraically integrable complex planar vector fields and singular algebraic foliations on surfaces, while Subsection \ref{Hirzebruch} does the same with the basics about Hirzebruch surfaces.  Keeping in mind the case of the projective plane (see \cite{Campiol} for instance), a foliation on a Hirzebruch surface can be given in a simple way as an affine vector field,  or an affine differential $1$-form, in four variables, where the coefficients are bigraded homogeneous polynomials which satisfy certain conditions with respect to radial vector fields, or Euler-type conditions (see the first part of Subsection \ref{sfh}).

\textit{Algorithm} \ref{algoritmo} considers a planar vector field $X$ (in the variables $x$ and $y$) and extends $X$ to a bigraded homogeneous affine $1$-form defining a foliation $\mathcal{F}^\delta_X$ on each Hirzebruch surface $\mathbb{F}_\delta$, where we take homogeneous coordinates $(X_0,X_1;Y_0, Y_1)$. These foliations are the key objects in our  \textit{first main result}, \textit{Theorem} \ref{gordo}, which proves that when $X \neq c \frac{\partial}{\partial y}$, $c \in \mathbb{C} \setminus \{0\}$, is algebraically integrable,  there exists a nonnegative integer $\delta_1$ forcing a very special local behaviour of the extended foliations $\mathcal{F}^\delta_X$ according to the position of $\delta$ with respect to $\delta_1$. More specifically, the point $(0,1;1,0)$ cannot be a dicritical singularity of $\mathcal{F}^{\delta_1}_X$ but it is always a dicritical singularity of $\mathcal{F}^\delta_X$ whenever $\delta < \delta_1$, and $(0,1;0,1)$ is the unique dicritical singularity of $\mathcal{F}^\delta_X$ belonging to the curve $X_0=0$ when $\delta > \delta_1$.

This result can be translated to the language of planar vector fields avoiding the use of Hirzebruch surfaces as we state in \textit{Corollary} \ref{ccc}.  It assigns to a vector field $X$ two families of infinitely many vector fields $\{X_{1,0}^\delta\}_{\delta \geq 0}$ and $\{X_{1,1}^\delta\}_{\delta \geq 0}$ and states that, when $X$ is algebraically integrable, the vector fields in these families must satisfy specific conditions of dicriticity at the origin and the points in the line $x=0$. In this way, {\it algebraic integrability of $X$ forces strong conditions on other derived planar vector fields giving rise to unknown necessary conditions for algebraic integrability.} It is worthwhile to add that our procedure discards the existence of rational first integrals for vector fields that do not satisfy the conditions in \cite[Corollary 5]{GineGrau} and  that, recently, differential Galois theory  has also be used for giving necessary conditions of integrability \cite{Acosta}. As before mentioned, a somewhat related problem (in the real setting) is the searching of algebraic limit cycles of planar polynomial vector fields. These limit cycles only exist when the vector field has no rational first integral  \cite{llibreopen} but the first integrals can be used for computing limit cycles when piecewise differential systems are considered \cite{Llibresym}.

The last section, Section \ref{La5}, contains our \textit{second main result}, \textit{Theorem} \ref{gordo2}. It considers an algebraically integrable vector field $X$ and that obtained by swapping their variables $X'$, and
uses their extensions to Hirzebruch surfaces to determine a region in $\mathbb{R}^{2}_{\geq 0}$ where all the pairs $(i,j)$ corresponding to monomials $x^i y^j$ with nonzero coefficient of the generic algebraic invariant curve $g$ of  $X$ are included. \textit{Corollary} \ref{El52} is deduced from Theorem \ref{gordo2} and when $\delta_1=0$ (respectively, $\delta'_1=0$) gives a bound on the degree of the first integral of $X$ depending on the value $\delta'_1$ (respectively, $\delta_1$) and the maximum degrees of the monomials $x^i$ and $y^j$ appearing in $g$ with nonzero coefficient.

\section{Preliminaries}
\label{prelim}

\subsection{Algebraically integrable complex  planar polynomial vector fields}

Let $\mathbb{C}[x,y]$ be the ring of polynomials in two variables $x$ and $y$ with complex coefficients and let $\mathbb{C}(x,y)$ be its quotient field. Consider a planar polynomial differential system
\begin{equation}\label{system}
\dot x=A(x,y),\;\;\; \dot y=B(x,y),
\end{equation}
where $A(x,y), B(x,y)\in \mathbb{C}[x,y]$ are coprime or, equivalently, the planar vector field
\begin{equation}\label{planar}
X=A(x,y)\frac{\partial }{\partial x}+B(x,y)\frac{\partial }{\partial y}.
\end{equation}
This planar vector field can also be determined by the differential $1$-form
$$\omega_X:=B(x,y)dx-A(x,y)dy.$$

 A \emph{rational first integral of  the system (\ref{system}) (or of $X$)} is a rational function $f\in \mathbb{C}(x,y)\setminus \mathbb{C}$ satisfying:
$$X(f)=A(x,y)\frac{\partial f}{\partial x}+B(x,y)\frac{\partial f}{\partial y} = 0$$
(or, alternatively, $\omega_X\wedge df = 0$).
We say that (\ref{system}) (or $X$) is \emph{algebraically integrable} if it admits a rational first integral $f$.

A rational function $h=h_1/h_2$ is said to be \emph{reduced} when $h_1$ and $h_2$ are coprime. The degree of a rational function $h$, $\deg(h)$, is the maximum of the degrees of $h_1$ and $h_2$. Moreover, $h\in\C(x,y)$ is said to be \emph{composite} if it can be written as $h=u\circ h'$, where $h'\in \mathbb{C}(x,y)\setminus \mathbb{C}$ and $u\in \mathbb{C}(t)$ with $\deg(u)\geq 2$. Otherwise $h$ is said to be \emph{noncomposite}.

An algebraically integrable differential system (\ref{system}) admits a noncomposite reduced rational first integral $f$. Any rational function of the form $h=u\circ f$, $u\in \mathbb{C}(t)\setminus \mathbb{C}$, is also a rational first integral of (\ref{system}); in addition, all the reduced rational first integrals of (\ref{system}) are of this form (see \cite[Theorem 10]{bost} for a proof). As a consequence, noncomposite reduced rational first integrals coincide with rational first integrals of minimal degree.
	
\begin{definition}\label{primitive}
A rational first integral of the system (\ref{system}) (or of $X$) is named \emph{primitive} if it is reduced and noncomposite.
\end{definition}

Let $h(x,y)$ be a nonzero polynomial in $\C[x,y]$. The algebraic curve with equation $h(x,y)=0$ is an \emph{invariant algebraic curve} of the system (\ref{system}) (or of $X$) if $X(h)=k(x,y)h(x,y)$ for some polynomial $k(x,y)$.

Let $f=\frac{f_1(x,y)}{f_2(x,y)}$ be a primitive rational first integral of (\ref{system}). Then the curves in $\C^2$ of the pencil $\alpha_1 f_1(x,y)+\alpha_2 f_2(x,y)=0$,  $(\alpha_1:\alpha_2)\in\mathbb{P}^1$, are reduced and irreducible with the exception of those corresponding to finitely many values in $\mathbb{P}^1$ (see, for instance, \cite[Chapter 2, Theorem 3.4.6]{Joua}). The
 solutions of (\ref{system}) are algebraic and are given by the irreducible components of the curves in this pencil (irreducible algebraic invariant curves).  Abusing the notation, in this paper, the expression $\alpha_1 f_1(x,y)+\alpha_2 f_2(x,y)$, regarded as a polynomial in $\mathbb{C}(\alpha_1,\alpha_2)[x,y]$, where $\alpha_1,\alpha_2$ are also considered variables, will be named the \emph{generic algebraic invariant curve} of $X$ (associated to $f$).

\subsection{Singular algebraic foliations on surfaces}
Many interesting results about algebraic integrability of (singular) planar vector fields come from the study of singular algebraic foliations on the projective plane. This happens  because these foliations are determined by homogeneous vector fields which extend the field of directions defined by the original planar vector field to the projective plane. In this article, we will use Hirzebruch surfaces instead of the projective plane.

We start by briefly recalling what a foliation is and, also, by providing some additional definitions and results we will use.

A (singular algebraic) \textit{foliation} \fol on a  smooth complex projective surface $S$ can be defined by a family of pairs  $\{(U_i,v_i)\}_{i\in I}$,  given by an open covering $\{U_i\}_{i\in I}$ of $S$ and nonvanishing vector fields $v_i$ on $\Theta_S(U_i)$, $\Theta_S$ being the tangent sheaf of $S$, such that for any three indices $i, j, k$ in $I$, the following equalities hold:
\begin{equation*}
\begin{array}{c}
v_i=g_{ij}v_j \text{ on }U_i\cap U_j, \mbox{ for some element } g_{ij}\in \mathcal{O}_S(U_i\cap U_j)^* \text{ and}\\
g_{ij}g_{jk}=g_{ik} \text{ on }U_i\cap U_j \cap U_k.
\end{array}
\end{equation*}
The \emph{singular set} $\Sing({\mathcal F})$ of $\mathcal F$ is the subset of $S$ such that $\Sing({\mathcal F})\cap U_i$ is the set of zeros of $v_i$ for all $i\in I$. The points in $\Sing({\mathcal F})$ are called \emph{singularities} of ${\mathcal F}$ and, when all the vector fields $v_i$ have isolated zeros, we say that ${\mathcal F}$ has \emph{isolated singularities} (and, in this case, $\Sing({\mathcal F})$ is discrete). All the foliations considered in this paper have isolated singularities.

The cocycle $(g_{ij})$ determines an invertible sheaf $\mathcal{L}^*$ of $S$ and the family $\{(U_i,v_i)\}_{i\in I}$ gives rise to a global section either in $H^{0}(S,\Theta_S \otimes \mathcal{L}^*)$ or in $H^0(S, \mathrm{Hom}_{\mathcal{O}_S}(\mathcal{L}, \Theta_S))$, where $\mathcal{L}$ is the dual of $\mathcal{L}^*$. Since two global sections define the same foliation whenever they differ by a nonzero scalar, \fol can be seen as a class $[\mathfrak{f}]$ in $\mathbb{P}(H^{0}(S,\Theta_S \otimes \mathcal{L}^*))$ of a global section $\mathfrak{f} \in H^{0}(S,\Theta_S \otimes \mathcal{L}^*)$. 

A foliation \fol can also be defined by using $1$-forms and then it is given by a family $\{(U_i,\omega_i)\}_{i\in I}$, where $\{U_i\}_{i\in I}$ is an open covering of $S$, $\omega_i\in\Omega_S^1(U_i)$ is a nonzero regular differential $1$-form on $U_i$ and, for indices $i, j$ and $k$ in $I$, it holds that
\begin{equation} \label{ec_foliation_gen}
\begin{array}{c}
\omega_i=f_{ij}\omega_j \text{ on }U_i\cap U_j, \mbox{ for some element } f_{ij}\in \mathcal{O}_S(U_i\cap U_j)^* \text{ and}\\
\omega_{ij}\omega_{jk}=\omega_{ik} \text{ on }U_i\cap U_j \cap U_k.
\end{array}
\end{equation}

\section{Foliations on Hirzebruch surfaces}\label{hirz}
We desire to study singular complex planar polynomial vector fields through singular foliations on Hirzebruch surfaces. For this reason we start with a brief introduction to this class of surfaces.

\subsection{Hirzebruch surfaces}
\label{Hirzebruch}
Let $\delta$ be a nonnegative integer. For each value $\delta$, there is a rational surface $\mathbb{F}_\delta$ named the  $\delta$th complex Hirzebruch surface. It is defined as the projectivization of the sheaf $\mathcal{O}_{\mathbb{P}^1}\oplus \mathcal{O}_{\mathbb{P}^1}(\delta)$:
\[
\fd:=\mathbb{P}\left(\mathcal{O}_{\mathbb{P}^1}\oplus \mathcal{O}_{\mathbb{P}^1}(\delta) \right),\]
$\mathbb{P}^1$ being  the complex projective line \cite[Chapter V, Corollary
2.13]{Har}.

$\fd$ is a ruled surface which has the structure of a toric variety. Thus, $\fd$ is the quotient of the Cartesian product $\left(\mathbb{C}^2\setminus \{\mathbf{0}\}\right)\times \left(\mathbb{C}^2\setminus \{\mathbf{0}\}\right)$ by the following action on the algebraic torus $\left(\mathbb{C}\setminus \{0\}\right)\times \left(\mathbb{C}\setminus \{0\}\right)$:
\[
(\lambda,\mu):(X_0,X_1;Y_0,Y_1) \mapsto (\lambda X_0,\lambda X_1; \mu Y_0, \lambda^{-\delta} \mu Y_1),
\]
where $(X_0,X_1;Y_0,Y_1)$ are coordinates in $\left(\mathbb{C}^2\setminus \{\mathbf{0}\}\right)\times \left(\mathbb{C}^2\setminus \{\mathbf{0}\}\right)$.

The homogeneous coordinate ring of $\fd$ is the polynomial ring in four variables $\mathbb{C}[X_0,X_1,Y_0,Y_1]$, where the variables are bigraded as follows: $\deg X_0 = \deg X_1 =(1,0)$, $\deg (Y_0)= (0,1)$  and $\deg Y_1 = (-\delta, 1)$.

The above given action $\mapsto$ provides a surjective map
\[
\varphi: \left(\mathbb{C}^2\setminus \{\mathbf{0}\}\right)\times \left(\mathbb{C}^2\setminus \{\mathbf{0}\}\right) \rightarrow \mathbb{F}_\delta,
\]
which allows us to consider the following four affine open sets of $\fd$:
\[
U_{ij}:= \left\{ \varphi(X_0,X_1;Y_0,Y_1) \; | \; X_i \neq 0 \; \mbox{ and } \; Y_j \neq 0 \right\},
\]
where $0 \leq i, j \leq 1$, giving rise to an open cover of $\fd$.

Now, on $U_{00}$,
$$\varphi(X_0,X_1;Y_0,Y_1) = \varphi\left(1,\frac{X_1}{X_0},1,\frac{X_0^\delta Y_1}{Y_0}\right),$$
and one can identify $U_{00}$ with $\mathbb{C}^2$ after setting $\varphi(X_0,X_1;Y_0,Y_1)=(x_{00},y_{00})$, where $x_{00}=\frac{X_1}{X_0}$ and $y_{00}=\frac{X_0^\delta Y_1}{Y_0}$. Analogous procedures can be carried out for identifying the remaining sets $U_{ij}$ with $\mathbb{C}^2$ and one gets the change of coordinate maps, in the overlaps of the sets $U_{ij}$, which provide the global structure of $\fd$. For instance, the change of coordinates in the intersection $U_{00} \cap U_{10}$ is
\[
\left(x_{00},y_{00}\right) \rightarrow \left(\frac{1}{x_{00}}, x_{00}^\delta y_{00}\right) = \left(x_{10},y_{10}\right),
\]
where $\left(x_{10}=\frac{X_0}{X_1},y_{10}=\frac{X_1^\delta Y_1}{Y_0}\right)$ are local coordinates in the set $U_{10}$.

To conclude this subsection, we notice that the divisor class group of $\fd$ \cite{Har} is generated by the linear equivalence classes of two divisors $F$ and $M$ such that $F^2=0$, $M^2= \delta$ and $F \cdot M=1$.  Moreover, the effective classes are those of divisors of the form $d_1 F + d_2 M$, where $d_1$ and $d_2$ are integers such that $d_1 + \delta d_2 \geq 0$ and $d_2 \geq 0$. Finally, the nonzero global sections $\mathcal{O}_{\mathbb{F}_\delta}(d_1F+d_2M)$ of effective divisors $d_1 F + d_2 M$ correspond to bigraded homogeneous polynomials $F(X_0,X_1;Y_0,Y_1) \in \mathbb{C}[X_0,X_1,Y_0,Y_1]$ of bidegree $(d_1,d_2)$. These polynomials are sums of monomials $X_0^a  X_1^b  Y_0^c Y_1^d$ such that $a+b-\delta d = d_1$ and $c +d = d_2$.

Next we study foliations on Hirzebruch surfaces showing that they have simple bigraded expressions.

\subsection{Singular foliations on Hirzebruch surfaces. Bigraded expressions and reduction of singularities}
\label{sfh}
In this subsection, we show that a singular foliation on a Hirzebruch surface can be given by an affine vector field (or an affine differential $1$-form) determined by  suitable bigraded homogeneous polynomials in the variables $X_0,X_1,Y_0,Y_1$. This will ease our study of planar vector fields through foliations on Hirzebruch surfaces.

Let $\mathcal{F} = [\mathfrak{f}]$ be a singular foliation (with isolated singularities) on a Hirzebruch surface. Then it is given by $\mathfrak{f} \in H^{0}(\mathbb{F}_\delta,\Theta_{\mathbb{F}_\delta} \otimes \mathcal{L}^*)$, where $\mathcal{L}=\mathcal{O}_{\mathbb{F}_\delta}(-d_1 F -d_2 M)$, $d_2 \geq 0$ and $d_1 \geq 0$ when $\delta=0$ or $d_1 \geq -1$ otherwise (see \cite{GalMonOliv}).

By \cite[Section 3]{GalMonOliv}, $\mathcal{F} $ can be given through an affine vector field or through an affine $1$-form. Indeed, on the one hand, $\mathcal{F}$  is uniquely determined by an affine vector field $V$ as follows:
\[
V=V_0 \frac{\partial}{\partial X_0} + V_1 \frac{\partial}{\partial X_1} + W_0 \frac{\partial}{\partial Y_0} + W_1 \frac{\partial}{\partial Y_1},\]
where $V_0$, $V_1$, $W_0$ and $W_1$ are bigraded homogeneous polynomials without nonconstant common factors (not all of them equal to 0),
$V_0,V_1\in H^0\left(\fd,\mathcal{O}_{\fd}((d_1+1)F+ d_2 M)\right)$, $W_0 \in H^0\left(\fd,\mathcal{O}_{\fd}(d_1 F + (d_2 + 1)M)\right)$ and $W_1\in H^0\left(\fd, \mathcal{O}_{\fd}((d_1-\delta)F + (d_2 + 1)M)\right)$, up to the addition of multiples of the radial vector fields \begin{gather*}
R_1:=X_0 \frac{\partial}{\partial X_0} + X_1 \frac{\partial}{\partial X_1} - \delta Y_1 \frac{\partial}{\partial Y_1}\; \text{ and}\\
R_2:=Y_0 \frac{\partial}{\partial Y_0} + Y_1 \frac{\partial}{\partial Y_1}.
\end{gather*}

On the other hand, $\mathcal{F}$  is also uniquely determined by an affine  differential $1$-form $\Omega$:
\[
\Omega=A_0 dX_0 + A_1  dX_1 + B_0 dY_0 + B_1 dY_1,
\]
where $A_0$, $A_1$, $B_0$ and $B_1$ are bigraded homogeneous polynomials without nonconstant common factors (not all of them equal to 0), $$A_0,A_1 \in H^0\left(\fd, \mathcal{O}_{\fd}((d_1 - \delta + 1)F+ (d_2 + 2)M)\right),$$ $$B_0 \in H^0\left(\fd, \mathcal{O}_{\fd}((d_1 - \delta + 2)F +( d_2 + 1)M)\right)$$ and $B_1 \in H^0\left(\fd, \mathcal{O}_{\fd}((d_1+2)F+ (d_2 + 1)M)\right)$, which satisfy the following two conditions, called Euler-type conditions:
\begin{gather*}
\Omega(R_1)=A_0 X_0+A_1 X_1 - \delta B_1 Y_1 =0 \text{ and}\\
\Omega(R_2)=B_0 Y_0 + B_1 Y_1=0.
\end{gather*}

Notice that, considering the before defined open cover of  $\{U_{ij}\}_{0 \leq i, j \leq 1}$, affine vector fields or affine $1$-forms allow us to get families $\{(U_i,v_i)\}_{i\in I}$ or $\{(U_i,\omega_i)\}_{i\in I}$ defining $\mathcal{F}$. We will preferably use affine $1$-forms.

Our foliations have isolated singularities. One of the main techniques for treating these singularities is the blowup of the surface (corresponding to the foliation) at the singular points of the foliation and the consideration of the strict transform of the foliation on the blown up surface. Some references about the blowup procedure for singular foliations on smooth projective surfaces are \cite{Seiden,Dumor,AFJ} (see also \cite{Cas}). Roughly speaking, blowing up a smooth complex surface $S$, at a point $p$, consists of replacing that point $p$ by a complex projective line regarded as the set of limit directions at $p$. The obtained surface is defined by local charts and usually denoted by Bl$_p(S)$.

As explained, a singular foliation $\mathcal{G}$ on $S$ can be locally given by  differential $1$-forms. These forms, after blowing up at a singularity $p$, define differential $1$-forms on affine charts of Bl$_p(S)$ which, after gluing, give rise to a foliation on the surface  Bl$_p(S)$ named the strict transform $\tilde{\mathcal{G}}$ of $\mathcal{G}$. A detailed description of this procedure is given in \cite[Section 4]{fergal1}.

Assume that $\mathcal{G}$ is given at $p$, in local coordinates $x$ and $y$, by a local differential $1$-form $\omega =A(x,y) dx + B(x,y) dy$.  Then, we define the multiplicity of  $\mathcal{G}$ at $p$ as the nonnegative integer $m$  corresponding to the first nonvanishing jet $\omega_m :=a_m(x,y) dx + b_m(x,y) dy$ of $\omega$. Notice that $a_m(x,y)$ and $b_m(x,y)$ are homogeneous polynomials in two variables of degree $m$ and that $p$ is a singularity of $\mathcal{G}$ if and only if $m\geq 1$. If this is the case, defining $d(x,y):= x a_m(x,y)+y b_m(x,y) $, we say that $p$ is a {\it terminal dicritical singularity} of $\mathcal{G}$ whenever $d(x,y)$ is the zero polynomial. In addition, $p$ is a {\it simple singularity} if $m=1$ and the matrix
\[
 \left(
    \begin{array}{cc}
      \frac{\partial b_1}{\partial x} & \frac{\partial b_1}{\partial y} \\
     - \frac{\partial a_1}{\partial x} & - \frac{\partial a_1}{\partial y}\\
    \end{array}
  \right)
\]
has  two eigenvalues $\lambda_1, \lambda_2$ such that the product $\lambda_1 \lambda_2$ does not vanish and their quotient is not a positive rational number, or $\lambda_1 \lambda_2 = 0$ and $\lambda_1^2 + \lambda_2^2$ does not vanish. Nonsimple singularities are named to be {\it ordinary}. Notice that a terminal dicritical singularity is an ordinary singularity.

The singularities of a foliation $ \mathcal{G}$ with isolated singularities can be reduced by blowing up at the ordinary singular points of $\mathcal{G}$ and at those belonging to its successive strict transforms. The following result summarizes well-known facts on reduction of foliations and terminal dicritical singularities (see \cite{Seiden,Bru} and \cite[Theorem 1 and Proposition 1]{fergal1}).

\begin{theorem}
\label{jmaa}
Let $\mathcal{G}$ be a singular foliation with isolated singularities on a smooth projective surface $S$. Then:
\begin{enumerate}
\item There is a sequence of finitely many point blowups, $\pi: Z \rightarrow S$, such that the strict transform of $\mathcal{G}$  on $Z$ has no ordinary singularity.
\item A singularity $p$ of $\mathcal{G}$ is not terminal dicritical if and only if the exceptional divisor of the surface {\rm Bl}$_p(S)$ is invariant by the strict transform of $\mathcal{G}$ on {\rm Bl}$_p(S)$.
\end{enumerate}

\end{theorem}

Let $p$ be a point in $S$. The points in the exceptional divisor $E_{p}$ obtained after blowing up at $p$ are named points in the \emph{first infinitesimal neighbourhood} of $p$. Assuming that we blow up at some points in $E_p$, the points in the new created exceptional divisors are in the \emph{second infinitesimal neighbourhood} of $p$. Inductively one defines the \emph{$k$th infinitesimal neighbourhood} of $p$, $k\geq 1$. A point $q$ is \emph{infinitely near} $p$ if either $q=p$ or it belongs to some $k$th infinitesimal neighbourhood of $p$. Finally, a point $q$ is \emph{infinitely near} $S$ if it is infinitely near a point in $S$.

In addition, given two infinitely near $S$ points $p$ and $q$, $q$ \emph{is proximate to} $p$  if $q$ belongs to $E_p$ or to any of its strict transforms. When $q$ is proximate to $p$ and $q\notin E_p$, $q$ is named \emph{satellite} and, otherwise, it is called \emph{free}.

Let $\mathcal{G}$ be a foliation and $\pi$ a map as in Theorem \ref{jmaa}. The set $\mathcal{C}_{\mathcal G}=\{p_1,\ldots,p_n\}$ of blowup centers of $\pi$ is called the \emph{singular configuration} of $\mathcal{G}$, and its elements \emph{infinitely near ordinary singularities} of $\mathcal{G}$. Notice that $\mathcal{C}_{\mathcal G}$ is formed by the ordinary singularities of $\mathcal{G}$ and those of the successive strict transforms of $\mathcal{G}$ by the sequence of blowups $\pi$. An \emph{infinitely near ordinary singularity} $p_i \in \mathcal{C}_{\mathcal{G}}$ is named an (infinitely near) {\it dicritical singularity} if there is a point $p_j \in \mathcal{C}_{\mathcal G}$ which is infinitely near $p_i$ and such that $p_j$ is a terminal dicritical singularity of the strict transform of $\mathcal{G}$ on the surface containing $p_j$. 

\subsection{Extending a planar vector field to a foliation on a Hirzebruch surface}

In this section we describe how any complex planar vector field can be extended to a singular algebraic foliation (with isolated singularities) on any Hirzebruch surface $\mathbb{F}_{\delta}$. For this purpose, we consider a complex planar vector field  $X$ as in (\ref{planar})  and the following algorithm, whose input is a nonnegative integer $\delta$ and a differential $1$-form $A(x,y)dx + B(x,y)dy$ (with $A(x,y)$ and $B(x,y)$ in $\mathbb{C}[x,y]$ and coprime) defining $X$, and whose ouput are four bigraded homogeneous polynomials $A_{\delta,0}, A_{\delta,1}, B_{\delta,0}, B_{\delta,1}\in \mathbb{C}[X_0,X_1,Y_0,Y_1]$ of suitable bidegrees.

\begin{algorithm}\label{algoritmo}\mbox{}

\textbf{Input:} A pair $(\delta,\omega)$, where $\delta\in \mathbb{Z}_{\geq 0}$ and $\omega=A(x,y)dx+ B(x,y)dy$ ($A(x,y), B(x,y) \in \mathbb{C}[x,y]$ coprime).

\textbf{Output:} $A_{\delta,0}, A_{\delta,1}, B_{\delta,0}, B_{\delta,1} \in\C[X_0,X1,Y_0,Y_1]$.\\

\begin{enumerate} [label=(\arabic*)]

\item Write the rational functions $A\left(\frac{X_1}{X_0}, \frac{X_0^{\delta} Y_1}{Y_0} \right)$ and $B\left(\frac{X_1}{X_0}, \frac{X_0^{\delta} Y_1}{Y_0} \right)$ as reduced rational fractions $\frac{A_{\delta, 1}}{X_0^{\alpha_1} Y_0^{\alpha_2}}$ and $\frac{B_{\delta, 1}}{X_0^{\beta_1} Y_0^{\beta_2}}$, respectively, where $(\alpha_1,\alpha_2),(\beta_1,\beta_2)\in \mathbb{Z}_{\geq 0}\times \mathbb{Z}_{\geq 0}$ and $A_{\delta, 1}$ and $B_{\delta, 1}$ are bigraded homogeneous polynomials in $\C[X_0,X_1,Y_0,Y_1]$ of respective bidegrees $(\alpha_1,\alpha_2)$ and $(\beta_1,\beta_2)$.

\item Let $m_1:=\alpha_1-\beta_1+1+\delta$. If $m_1>0$, then $B_{\delta, 1}:=X_0^{m_1}B_{\delta, 1}$; otherwise, $A_{\delta, 1}:=X_0^{-m_1}A_{\delta, 1}$.

\item Let $m_2:=\alpha_2-\beta_2-1$. If $m_2>0$, then $B_{\delta, 1}:=Y_0^{m_2}B_{\delta, 1}$; otherwise, $A_{\delta, 1}:=Y_0^{-m_2}A_{\delta, 1}$.

\item Let $b:=0$ if $Y_0$ divides $B_{\delta, 1}$, and $b:=1$ otherwise. Set $B_{\delta, 1}:= Y_0^b B_{\delta, 1}$ and $A_{\delta, 1}:=Y_0^{b}A_{\delta, 1}$.

\item Let $a:=0$ if $X_0$ divides $\delta Y_1 B_{\delta, 1}-X_1 A_{\delta, 1}$ and $a:=1$ otherwise. Set $A_{\delta, 1}:=X_0^a A_{\delta, 1}$ and $B_{\delta, 1}:=X_0^a B_{\delta, 1}$.

\item Set $A_{\delta, 0}:=\frac{\delta Y_1 B_{\delta, 1}-X_1 A_{\delta, 1}}{X_0}$ and $B_{\delta, 0}:=\frac{-Y_1 B_{\delta, 1}}{Y_0}$.
\end{enumerate}

\end{algorithm}

\begin{lemma}\label{lem}
Fix $\delta\in \mathbb{Z}_{\geq 0}$. Let $\omega_X=A(x,y)dx+B(x,y)dy$ be a differential $1$-form defining a planar vector field $X$, and let $A_{\delta, 0}, A_{\delta, 1}, B_{\delta, 0}$ and $B_{\delta, 1}$ be the polynomials of $\mathbb{C}[X_0,X_1,Y_0,Y_1]$ obtained as the output of Algorithm \ref{algoritmo} from the input given by the pair $(\delta,\omega_X)$. Then $A_{\delta, 0},A_{\delta, 1},B_{\delta, 0}$ and $B_{\delta, 1}$ are bigraded homogeneous polynomials with respective bidegrees $(d_1-\delta+1,d_2+2)$, $(d_1-\delta+1,d_2+2)$, $(d_1-\delta+2,d_2+1)$ and $(d_1+2,d_2+1)$ for some integers $d_1,d_2$. Moreover they satisfy the equalities
\begin{equation}\label{equalities}
X_0A_{\delta, 0}+X_1A_{\delta, 1}-\delta Y_1B_{\delta, 1}=0\;\;\mbox{ and }\;\; Y_0B_{\delta, 0}+Y_1B_{\delta, 1}=0
\end{equation}
and have no nonconstant common factor.

\end{lemma}

\begin{proof}

Notice that the polynomials $A_{\delta, 1}$ and $B_{\delta, 1}$ obtained in Step (1) of Algorithm \ref{algoritmo} are coprime and have respective bidegrees $(\alpha_1,\alpha_2)$ and $(\beta_1,\beta_2)$. A straightforward but tedious study of the different possibilities that may appear in Algorithm \ref{algoritmo} shows the existence of integers $d_1,d_2$ such that the bidegree of $A_{\delta, 1}$ is $(d_1-\delta+1,d_2+2)$ and the bidegree of $B_{\delta, 1}$ is $(d_1+2,d_2+1)$. As an example: if $m_1,m_2>0$ (in Steps (2) and (3)), $b=0$ in Step (4) and $a=1$ in Step (5); then $A_{\delta, 1}$ and $B_{\delta, 1}$ have respective bidegrees $(\alpha_1+1,\alpha_2)$ and $(\alpha_1+2+\delta,\alpha_2-1)$. Hence $d_1=\alpha_1+\delta$ and $d_2=\alpha_2-2$ in this case.

The polynomials $A_{\delta, 1}$ and $B_{\delta, 1}$ obtained after applying the steps from (1) to (5) satisfy that $X_0$ (respectively, $Y_0$) divides $\delta Y_1 B_{\delta, 1}-X_1A_{\delta, 1}$ (respectively, $B_{\delta, 1}$). Therefore the rational functions $A_{\delta, 0}$ and $B_{\delta, 0}$ defined in Step (6) are polynomials and their bidegrees coincide with those given in the statement. In addition, Equalities (\ref{equalities}) hold trivially.

It is easily derived from the algorithm that the only two possible common factors of the output polynomials are $X_0$ and $Y_0$. Let us see that none of them can be such a common factor. The polynomials $A_{\delta,1}$ and $B_{\delta,1}$ obtained in Step (1) do not share factors with $X_0Y_0$. After steps (2) and (3), at most one of them ($A_{\delta,1}$ and $B_{\delta,1}$) has $X_0$ (respectively, $Y_0$) as a factor. On the one hand, in Step (4) we ensure that either $Y_0$ does not divide $A_{\delta,1}$, or $Y_0$ divides $B_{\delta,1}$ but $Y_0^2$ does not (what implies that $Y_0$ does not divide $B_{\delta,0}$ after Step (6)). On the other hand, in Step (5) we force $X_0$ to divide $\delta Y_1 B_{\delta, 1}-X_1A_{\delta, 1}$ (but $X_0^2$ does not); then, after  Step (6), $X_0$ does not divide $A_{\delta,0}$.
\end{proof}

Consider a nonnegative integer $\delta$ and identify the affine plane $\mathbb{C}^2$ with the open subset $U_{00}$ of $\mathbb{F}_{\delta}$. Then, as a consequence of Section \ref{sfh} and Lemma \ref{lem}, we deduce the following result.

\begin{proposition}
\label{la33}
Let $\delta$ be a nonnegative integer and  $\omega_X=A(x,y)dx+B(x,y)dy$  a differential $1$-form defining a complex planar polynomial vector field $X$. Let $$A_{\delta, 0},A_{\delta, 1},B_{\delta, 0},B_{\delta, 1}\in \mathbb{C}[X_0,X_1,Y_0,Y_1]$$ be the output of Algorithm \ref{algoritmo} when its input is the pair $(\delta, \omega_X)$. Then the bigraded homogeneous affine differential $1$-form
$$\Omega_{\delta}:=A_{\delta, 0}dX_0+A_{\delta, 1}dX_1+B_{\delta, 0}dY_0+B_{\delta, 1}dY_1$$
defines a singular algebraic foliation on the Hirzebruch surface $\mathbb{F}_{\delta}$, with isolated singularities, whose restriction to the open set $U_{00}$ gives the $1$-form in two variables that determines the vector field $X$.

\end{proposition}

The foliation obtained from the pair $(\delta, \omega_X)$ by Proposition \ref{la33} is called the \emph{extension of the vector field $X$} to the Hirzebruch surface $\mathbb{F}_{\delta}$ and it is denoted by ${\mathcal F}_X^{\delta}$.\\

If $f=\frac{f_1(x,y)}{f_2(x,y)}$ is a reduced rational function and $\delta\in \mathbb{Z}_{\geq 0}$, then there exist two coprime bigraded homogeneous polynomials $F_1,F_2\in \mathbb{C}[X_0,X_1,Y_0,Y_1]$ of the same bidegree such that the following equality of rational functions holds:
$$f(X_1/X_0, X_0^\delta Y_1/Y_0)=\frac{F_1}{F_2}.$$
By \cite[Proposition 1.6]{Scardua}, a planar vector field $X$ has $f$ as a rational first integral if and only if the function $F_1/F_2$ is a rational first integral of the foliation ${\mathcal F}_X^{\delta}$.

\section{Necessary conditions for algebraic integrability}
This section is devoted to study the behaviour of algebraically integrable complex planar polynomial vector fields $X$ through their extensions to foliations ${\mathcal F}_X^{\delta}$ on Hirzebruch surfaces. We show in Theorem \ref{gordo} that the points with coordinates $(0,1;0,1)$ (respectively, $(0,1;1,0)$) in each surface $\fd$ are dicritical singularities of ${\mathcal F}_X^{\delta}$ whenever $\delta > \delta_1$ (respectively, $\delta < \delta_1$) for a fixed nonnegative integer $\delta_1$ which is the minimum nonnegative integer such that $(0,1;1,0)$ is not a dicritical singularity of ${\mathcal F}_X^{\delta}$. This result can be reformulated in terms of planar vector fields depending on a nonnegative integer parameter which gives rise to a new technique for discarding the existence of a rational first integral of a vector field (see Corollary \ref{ccc}). We start with a lemma which we will use in the proof of the forthcoming Theorem \ref{gordo}.

\begin{lemma}\label{llll}
Let $X$ be an algebraically integrable complex planar vector field. Let $f=\frac{f_1(x,y)}{f_2(x,y)}$ be a primitive rational first integral of $X$ and $g(x,y)=\alpha f_1(x,y)+\beta f_2(x,y)\in \mathbb{C}(\alpha,\beta)[x,y]$ the associated generic algebraic invariant curve of $X$. Then $X\neq c \frac{\partial}{\partial y}$ for all $c\in \mathbb{C}\setminus \{0\}$ if and only if $g(x,y)\not\in \mathbb{C}(\alpha,\beta)[x]$.
\end{lemma}

\begin{proof}

$X=c \frac{\partial}{\partial y}$ for some $c\in \mathbb{C}\setminus \{0\}$ if and only if the foliation on $\mathbb{C}^2$ defined by $X$ is determined by the $1$-form $\omega_X=dx$. This means that the function $x$ is a first integral of this foliation, that is, $f_1(x,y)$ and $f_2(x,y)$ are polynomials in $\mathbb{C}[x]$ of degree 1.  This is equivalent to say that $g(x,y)\in \mathbb{C}(\alpha,\beta)[x]$ because the polynomial of $\mathbb{C}[x,y]$ obtained after replacing, in $g(x,y)$, $\alpha$ and $\beta$ by general complex numbers, must be irreducible.

\end{proof}

\begin{theorem}\label{gordo}
Let $X$ be an algebraically integrable complex planar polynomial vector field such that $X\neq c \frac{\partial}{\partial y}$ for all $c\in \mathbb{C}\setminus \{0\}$. For each $\delta\in \mathbb{Z}_{\geq 0}$, consider the foliation ${\mathcal F}_X^\delta$ given by the extension of $X$ to the Hirzebruch surface $\mathbb{F}_{\delta}$. Then, there exists a nonnegative integer $\delta_1$ satisfying the following conditions:
\begin{itemize}
\item[(i)]  For all integers $\delta$ such that $\delta> \delta_1$, the point $(0,1;0,1)\in\mathbb{F}_{\delta}$ is the unique dicritical singularity of ${\mathcal F}_X^\delta$ belonging to the curve $X_0=0$.

\item[(ii)] For all nonnegative integer $\delta$ such that $\delta < \delta_1$, the point $(0,1;1,0)\in \mathbb{F}_{\delta}$ is a dicritical singularity of ${\mathcal F}_X^\delta$.

\item[(iii)] The point $(0,1;1,0)\in \mathbb{F}_{\delta_1}$ is not a dicritical singularity of ${\mathcal F}_X^{\delta_1}$.

\end{itemize}

\end{theorem}

\begin{proof}

Let $f=\frac{f_1(x,y)}{f_2(x,y)}$ be a primitive rational first integral of $X$. Then the associated generic algebraic invariant curve of $X$ is $g(x,y)=\alpha f_1(x,y)+\beta f_2(x,y)\in \mathbb{C}(\alpha,\beta)[x,y]$. Let us write
$g(x,y)=\sum g_{ij}x^i y^j$, where the coefficients $g_{ij}$ are homogeneous linear polynomials in $\alpha, \beta$. Let $d_x$ (respectively, $d_y$) be the degree in the variable $x$ (respectively, $y$) of $g(x,y)$, that is, the degree of $g$ when it is regarded as a polynomial in $x$ (respectively, $y$) with coefficients in $\mathbb{C}(\alpha,\beta,y)$ (respectively, $\mathbb{C}(\alpha,\beta,x)$). Denote by $d_x^0$ (respectively, $d_y^0$) the degree of $g(x,0)$ (respectively, $g(y,0)$). Notice that $d_y>0$ by Lemma \ref{llll}.

We can write $g(x,y)$ as the sum of four polynomials $A, B, C$ and $D$ (with variables $x,y$ and coefficients in $\mathbb{C}(\alpha,\beta)$) as showed in the following displayed formula:
\begin{equation}\label{rrr}
g(x,y)=\underbrace{\sum_{i=0}^{d_x^0}g_{i0}x^i}_{=A}+ \underbrace{\sum_{j=1}^{d_y^0}g_{0j}y^j}_{=B}+ \underbrace{\sum_{{\tiny{\begin{array}{l}
1 \leq i\leq d_x^0
\end{array}}}}^{}\sum_{j=1}^{d_y'}g_{ij}x^iy^j}_{=C}+ \underbrace{\sum_{{\tiny{\begin{array}{l}
i>d_x^0
\end{array}}}}^{}\sum_{j=1}^{d_y''}g_{ij}x^iy^j}_{=D}.
\end{equation}

Denote by $\mathrm{Coeff}(h)$ the set of nonzero coefficients $h_{ij}$ of a polynomial $$h(x,y) = \sum h_{ij} x^i y^j \in \mathbb{C}(\alpha,\beta)[x,y].$$ Also, consider the following set of nonnegative rational numbers:
\begin{equation}
\label{gamma}
\Gamma=\left\{\frac{i-d_x^0}{j} \; | \; j>0\text{ and } g_{ij}\in \mathrm{Coeff}(g)\right\}\cap \Q_{\geq 0}.
\end{equation}

Let $\delta$ be an arbitrary nonnegative integer. Consider the Hirzebruch surface $\mathbb{F}_{\delta}$ and identify $\mathbb{C}^2$ with the open set $U_{00}$ of $\mathbb{F}_{\delta}$ as showed in Subsection \ref{Hirzebruch}. Then, replacing in Equation (\ref{rrr}), $x$ by $X_1/X_0$ and $y$ by $X_0^\delta Y_1/Y_0$, and multiplying by suitable powers $X_0^a$ and $Y_0^b$, we obtain a irreducible bigraded homogeneous polynomial $G_{\delta}(X_0,X_1;Y_0,Y_1) \in \mathbb{C}(\alpha,\beta)[X_0,X_1,Y_0, Y_1]$ of bidegree $(a,b)$.
\begin{multline*}
G_{\delta}(X_0,X_1;Y_0,Y_1):=\\X_0^aY_0^b\cdot\left(g_{00}+ \sum_{i=1}^{d_x^0}g_{i0}\frac{X_1^i}{X_0^i}+\sum_{j=1}^{d_y^0}g_{0j}\frac{X_0^{\delta j} Y_1^j}{Y_0^j}+ \sum_{{\tiny{\begin{array}{l}
1 \leq i \leq d_x^0
\end{array}}}}^{}\sum_{j=1}^{d_y'}g_{ij}\frac{X_0^{\delta j - i} X_1^i Y_1^j}{Y_0^j} \right.\\
\left.+\sum_{{\tiny{\begin{array}{l}
i>d_x^0
\end{array}}}}^{}\sum_{j=1}^{d_y''}g_{ij}\frac{X_0^{\delta j - i} X_1^i Y_1^j}{Y_0^j} \right).
\end{multline*}
The polynomial $G_{\delta}$ is not divisible by neither $X_0$ nor $Y_0$, $b=d_y=\max\{d_y^0,d_y',d_y''\}> 0$ and $a=a'+d_x^0$, with  $a'\in\Z_{\geq 0}$. Therefore,
\begin{multline*}
G_{\delta}(X_0,X_1;Y_0,Y_1)=\\X_0^{a'}\cdot\left(g_{00} X_0^{d_x^0}Y_0^{d_y}+ \sum_{i=1}^{d_x^0}g_{i0}X_0^{d_x^0-i}X_1^i Y_0^{d_y}+\sum_{j=1}^{d_y^0}g_{0j}X_0^{\delta j+ d_x^0} Y_0^{d_y-j} Y_1^j\right.\\
\left.+ \sum_{{\tiny{\begin{array}{l}
1 \leq i\leq d_x^0
\end{array}}}}^{}\sum_{j=1}^{d_y'}g_{ij} X_0^{\delta j +d_x^0- i} X_1^i Y_0^{d_y-j} Y_1^j +\sum_{{\tiny{\begin{array}{l}
i>d_x^0
\end{array}}}}^{}\sum_{j=1}^{d_y''}g_{ij} X_0^{\delta j +d_x^0- i} X_1^i Y_0^{d_y-j} Y_1^j \right).
\end{multline*}
Notice that, in the above expression between parentheses, negative exponents may only appear in the last block of summations.\\

{\it Firstly let us assume that $\Gamma=\emptyset$.} This implies that $i< d_x^0$ for all $g_{ij}\in \mathrm{Coeff}(g)$; so $D=0$. Then $a'=0$ and
\begin{multline*}
G_{\delta}(X_0,X_1;Y_0,Y_1)= g_{00} X_0^{d_x^0}Y_0^{d_y}+ \sum_{i=1}^{d_x^0}g_{i0}X_0^{d_x^0-i}X_1^i Y_0^{d_y}+\sum_{j=1}^{d_y^0}g_{0j}X_0^{\delta j+ d_x^0} Y_0^{d_y-j} Y_1^j\\
+ \sum_{{\tiny{\begin{array}{l}
i< d_x^0
\end{array}}}}^{}\sum_{j=1}^{d_y'}g_{ij} X_0^{\delta j +d_x^0- i} X_1^i Y_0^{d_y-j} Y_1^j.
\end{multline*}

Notice that $d_x^0>0$ because otherwise $B=0$ and $C=0$, what implies that $g(x,y)=g_{00}$ (a contradiction because, by Lemma \ref{llll}, $d_y>0$). Therefore $g_{d_x^0 0}\neq 0$. This shows that the point $(0,1;0,1)$ is the unique point belonging to the intersection of the curves on $\mathbb{F}_\delta$ defined by  the equations $X_0=0$ and $G_{\delta}(X_0,X_1;Y_0,Y_1)=0$ or, equivalently, it is the unique dicritical singularity of ${\mathcal F}_{X}^\delta$ belonging to the curve defined by $X_0=0$ (independently of the value of $\delta$). In this case $\delta_1=0$ is the integer satisfying the conditions given in the statement.

{\it Let us assume now that $\Gamma\neq \emptyset$.} Under this assumption,  let us define $$k:=\max(\Gamma)$$ and distinguish the following three cases, depending on the value of $\delta$:\\

\noindent $\textbf{Case 1:}$ The set
\begin{equation}
\label{AA}
\Delta:=\{ g_{ij} \in \mathrm{Coeff}(G_{\delta}) \mid j>0 \;\mbox{ and } \delta j+d_x^0-i<0 \}
\end{equation}
is not empty.

The above condition shows that $\Delta\subseteq \mathrm{Coeff}(D)$ and $\delta<k$. Moreover, $$a'=-\min\{\delta j+d_x^0-i\mid g_{ij}\in \Delta\}.$$ Hence, the points of intersection between the curves on $\mathbb{F}_{\delta}$ defined by the equations $X_0=0$ and $G_{\delta}(X_0,X_1;Y_0,Y_1)=0$ are the points $(0,1;y_0,y_1)$ satisfying the following condition \[
\sum_{j=1}^{d_y''}g_{\delta j + a'+d_x^0,j} y_0^{d_y-j} y_1^j=0.\]
In particular $(0,1;1,0)$ belongs to that intersection and, as a consequence, {\it $(0,1;1,0)$ is a dicritical singularity} of ${\mathcal F}_X^\delta$.\\

\noindent $\textbf{Case 2:}$ The set $\Delta$ in (\ref{AA}) is empty and there exists $g_{lm}\in \mathrm{Coeff}(G_{\delta})$ such that $m >0$ and $\delta m+d_x^0-l=0$.

In this case, since $\Delta$ is empty, $a'=0$ and $\delta\geq k$; moreover, since $\delta=\frac{l-d_x^0}{m}\in \Gamma$, we conclude that $\delta=k$. If $d_x^0=0$, then $C=0$ and $g_{00}\neq 0$; hence $G_{\delta}(0,1;1,0)\neq 0$, that is, {\it $(0,1;1,0)\in \mathbb{F}_k$ is not a dicritical singularity of ${\mathcal F}_X^k$}. When $d_x^0\neq 0$, the same thing happens because $g_{d_x^0 0}\neq 0$.
\\

\noindent $\textbf{Case 3:}$ $\delta j+d_x^0-i>0$ for all $g_{ij}\in \mathrm{Coeff}(G_{\delta})$ such that $j>0$.

Then $a'=0$ and $\delta >k$, and we distinguish the following subcases:
\begin{itemize}
\item[(3.1)] If $d_x^0>0$, then $g_{d_x^0 0}\neq 0$ and $(0,1;0,1)$ is the unique point where the curves with equations $G_{\delta}(X_0,X_1;Y_0,Y_1)=0$ and $X_0=0$ meet. This means that {\it $(0,1;0,1)$ is a dicritical singularity of ${\mathcal F}_{X}^\delta$ and the unique one belonging to the curve  $X_0=0$}.

\item[(3.2)] If $d_x^0=0$, then $G_{\delta}$ has the following shape:
$$
G_{\delta}(X_0,X_1;Y_0,Y_1)=g_{00} Y_0^{d_y}+\sum_{j=1}^{d_y^0}g_{0j}X_0^{\delta j} Y_0^{d_y-j} Y_1^j+X_0 H,
$$
where $H\in \mathbb{C}(\alpha,\beta)[X_0,X_1,Y_0,Y_1]$. Since $\delta>k\geq 0$, it is clear that $g_{00}\neq 0$ (because, otherwise, $X_0$ would divide $G_{\delta}$) and then {\it $(0,1;0,1)$ is the unique dicritical singularity of ${\mathcal F}_X^{\delta}$ belonging to the curve $X_0=0$}.
\end{itemize}

Notice that cases 1, 2 and 3 correspond to the following situations: $\delta<k$, $\delta=k$ and $\delta>k$.\\

Finally, define $\delta_1:=\lceil k \rceil$ and let us see that this integer satisfies conditions $(i)$, $(ii)$ and $(iii)$ of the statement.

If $k$ is an integer, then cases 1 and 3 show that conditions $(i)$ and $(ii)$ are satisfied for $\delta_1=k$. Hence, it only remains to show that $(0,1;1,0)\in \mathbb{F}_k$ is not a dicritical singularity of ${\mathcal F}_X^k$; but the value $\delta=k$ corresponds to Case 2 and
$(0,1;1,0)$ is not a dicritical singularity of ${\mathcal F}_X^\delta$ in that case.

If $k$ is not an integer then any $\delta\in \mathbb{Z}_{\geq 0}$ satisfies either Case 1 or Case 3; this fact shows that conditions $(i)$, $(ii)$ and $(iii)$ hold.

\end{proof}

\begin{remark}

Let $X$ be a complex planar vector field satisfying the conditions of Theorem \ref{gordo}. Then, the value $\delta_1$ provided by that theorem is the minimum nonnegative integer $\delta$ such that the point $(0,1;1,0)\in \mathbb{F}_{\delta}$ is not a dicritical singularity of ${\mathcal F}_X^{\delta}$.

\end{remark}

For each $\delta\in \mathbb{Z}_{\geq 0}$, the point $(0,1;0,1)\in \mathbb{F}_{\delta}$ (respectively, $(0,1;1,0)$) belongs to the affine chart $U_{11}$ (respectively, $U_{10}$), and the curve of $\mathbb{F}_{\delta}$ with equation $X_0=0$ does not meet neither $U_{00}$ nor $U_{01}$.  These facts allow us to write Theorem \ref{gordo} in terms of the planar vector fields induced by the restriction of ${\mathcal F}_X^\delta$ to the charts $U_{10}$ and $U_{11}$. Therefore, Theorem \ref{gordo} can be reformulated without any reference to Hirzebruch surfaces as follows:

\begin{corollary}\label{ccc}
Let $X$ be an algebraically integrable complex planar polynomial vector field such that $X\neq c \frac{\partial}{\partial y}$ for all $c\in \mathbb{C}\setminus \{0\}$. For each $\delta\in \mathbb{Z}_{\geq 0}$, let $A_{\delta,0},A_{\delta,1}, B_{\delta,0}$ and $B_{\delta,1}$ be the polynomials of $\mathbb{C}[X_0,X_1,Y_0,Y_1]$ obtained as the output of Algorithm \ref{algoritmo} from the input given by the pair $(\delta,\omega_X)$. Consider the planar vector fields $X^\delta_{10}$ and $X^\delta_{11}$ defined, respectively, by the following differential $1$-forms:
$$\omega^\delta_{10}:=A_{\delta,0}(x,1,1,y)dx+ B_{\delta,1}(x,1,1,y)dy,\;\mbox{ and }$$ $$\omega^\delta_{11}:=A_{\delta,0}(x,1,y,1)dx+B_{\delta,0}(x,1,y,1)dy.$$
Let $\delta_1$ be the minimum nonnegative integer such that the origin $(0,0)$ is not a dicritical singularity of $X_{10}^{\delta_1}$. Then, for all $\delta>\delta_1$:
\begin{itemize}
\item[(a)] the origin $(0,0)$ is the unique dicritical singularity of $X^\delta_{11}$ in the line defined by $x=0$, and
\item[(b)] the vector field $X_{10}^\delta$ has no dicritical singularity in the line defined by $x=0$.
\end{itemize}

\end{corollary}

As a consequence of the above result we state the following corollary, which provides conditions forcing a planar vector field to be nonalgebraically integrable.

\begin{corollary}\label{ccc2}
Let $X$ be a complex planar polynomial vector field such that $X\neq c \frac{\partial}{\partial y}$ for all $c\in \mathbb{C}\setminus \{0\}$. For every $\delta\in \mathbb{Z}_{\geq 0}$, consider the planar vector fields $X_{10}^\delta$ and $X_{11}^\delta$ defined in Corollary \ref{ccc}. Let $\mathcal{N}$ be the set of nonnegative integers $\delta$ such that origin $(0,0)$ is not a dicritical singularity of $X_{10}^\delta$. When $\mathcal{N} \neq \emptyset$, set $\delta_1 := \min \, \mathcal{N}$. Then, $X$ is not algebraically integrable if at least one of the following conditions is satisfied:
\begin{itemize}
\item[(a)] $\mathcal{N}$ is empty.
\item[(b)] $\mathcal{N}$ is not empty and there exists a positive integer $\delta>\delta_1$ such that either the origin $(0,0)$ is not a dicritical singularity of $X^\delta_{11}$, or $(0,0)$ is a dicritical singularity of $X^\delta_{11}$ but not the unique one in the line defined by the equation $x=0$.
\item[(c)] $\mathcal{N}$ is not empty and there exists a positive integer $\delta>\delta_1$ such that $X^\delta_{10}$ has a dicritical singularity in the line defined by $x=0$.
\end{itemize}
\end{corollary}

The following example gives a complex planar vector field which is not algebraically integrable, fact that we deduce from Corollary \ref{ccc2}.\\

\begin{example}\label{ex1}
Let $X$ be the planar vector field defined by the differential $1$-form $$\omega=(xy+y^2+5x^3y)dx + (-x^2 - xy + y^3)dy.$$

We run Algorithm \ref{algoritmo} using as input the pair $(\delta, \omega)$. The output is
\begin{align*}
A_{\delta,0}=&-X_0^2 X_1^2 Y_0^4 Y_1 - 5 X_1^4 Y_0^4 Y_1 - X_0^{\delta+3} X_1 Y_0^3 Y_1^2 -
\delta  X_0^2 X_1^2 Y_0^4 Y_1 -\delta X_0^{\delta+3} X_1 Y_0^3 Y_1^2\\
&+ \delta X_0^{3 \delta+4} Y_0 Y_1^4,\\
A_{\delta,1}=&X_0^3 X_1 Y_0^4 Y_1 + 5 X_0 X_1^3 Y_0^4 Y_1 + X_0^{\delta+4} Y_0^3 Y_1^2,\\
B_{\delta,0}=&X_0^3 X_1^2 Y_0^3 Y_1 + X_0^{\delta +4} X_1 Y_0^2 Y_1^2 - X_0^{3 \delta+5} Y_1^4, \; \mathrm{and}\\
B_{\delta,1}=&-X_0^3 X_1^2 Y_0^4 -  X_0^{\delta+4} X_1 Y_0^3 Y_1 + X_0^{3\delta+5} Y_0 Y_1^3.
\end{align*}
The vector field $X_{10}^\delta$ introduced in Corollary \ref{ccc} is given by the differential $1$-form
$$
\omega_{10}^\delta=(-5y-(1+ \delta)x^2y-(1 + \delta)x^{\delta+3}y^2+\delta x^{3\delta+4} y^4)dx+(-x^3-x^{\delta+4}y+x^{3\delta+5}y^3)dy.
$$
On the one hand, the origin is a simple singularity of $X_{10}^\delta$ for all $\delta\in \mathbb{Z}_{\geq 0}$ and then we deduce that $\delta_1=0$. On the other hand, the vector field $X_{11}^1$ is defined by the differential $1$-form
$$
\omega_{11}^1=(-5y^4- 2x^2y^4-2x^4y^3+x^7y)dx+(x^3y^3+x^5y^2-x^8)dy.
$$
Now, if we reduce the singularity $(0,0)$ of $\omega_{11}^1$ by successive blowups to get at most simple singularities (see Subsection \ref{sfh}) we  see that the origin is not a dicritical singularity of $X_{11}^1$. Indeed, to reduce the singularity $(0,0)$ we have to blow up 17 infinitely near points $\{p_i\}_{i=1}^{17}$ which constitute a simple chain, where $p_2$ is proximate to $p_1$; $p_3, p_4$ and $p_5$ are proximate to $p_2$, and $p_i$ is proximate to $p_{i-1}$ for $6 \leq i \leq 17$. No point $p_i$ is terminal dicritical, therefore $(0,0)$ is not dicritical. As a consequence, $X$ is not algebraically integrable by Part (b) of Corollary \ref{ccc2}.
\end{example}

\begin{remark}
Corollary \ref{ccc2} allows us to discard the existence of a rational first integral for certain complex planar vector fields. Necessary conditions for algebraic integrability are given in \cite[Corollary 5]{GineGrau} but they can only be applied to differential forms $A(x,y) dx + B(x,y) dy$, where $A(x,y)$ and $B(x,y)$ have the same degree $n$ and their homogeneous components of degree $n$ are coprime. Example \ref{ex1} does not satisfy those conditions, proving that the necessary conditions for algebraic integrability given in Corollary \ref{ccc} are different from those in \cite{GineGrau}.
\end{remark}

The conditions for algebraic integrability given in Theorem \ref{gordo} (and Corollary \ref{ccc}) are necessary, but not sufficient, as the following example shows.

\begin{example}
Let $X$ be the complex planar vector field defined by the differential $1$-form $$\omega=(y + xy)dx + (1+xy^2+x^2)dy.$$
The output of Algorithm \ref{algoritmo} when the input is the pair $(1, \omega)$ is
\begin{align*}
A_{1,0}&=X_0 Y_0^3 Y_1 - X_1 Y_0^3 Y_1 + X_0^2 X_1 Y_0 Y_1^3,\\
A_{1,1}&=X_0 Y_0^3 Y_1 + X_1 Y_0^3 Y_1,\\
B_{1,0} &=-X_0^2 Y_0^2 Y_1 - X_1^2 Y_0^2 Y_1 - X_0^3 X_1 Y_1^3 \; \mathrm{and}\\
B_{1,1}&=X_0^2 Y_0^3 + X_1^2 Y_0^3 + X_0^3 X_1 Y_0 Y_1^2.
\end{align*}

and when the input is $(\delta \neq 1, \omega)$, it is
\begin{align*}
A_{\delta,0}&=-X_0 X_1 Y_0^3 Y_1 - X_1^2 Y_0^3 Y_1 +\delta X_0^2 Y_0^3 Y_1 +\delta X_1^2 Y_0^3 Y_1 +\delta X_0^{2 \delta+1} X_1 Y_0 Y_1^3,\\
A_{\delta,1} &=X_0^2 Y_0^3 Y_1 + X_0 X_1 Y_0^3 Y_1,\\
B_{\delta,0}&=-X_0^3 Y_0^2 Y_1 - X_0 X_1^2 Y_0^2 Y_1 - X_0^{2 \delta+2} X_1 Y_1^3 \; \mathrm{and}\\
B_{\delta,1} &=X_0^3 Y_0^3 + X_0 X_1^2 Y_0^3 + X_0^{2 \delta+2} X_1 Y_0 Y_1^2.
\end{align*}

These outputs define the foliations $\mathcal{F}^\delta_X$, $\delta \geq 0$.

For a start, $(0,1;1,0) \in {\mathbb F}_0$ is a terminal dicritical singularity of $\mathcal{F}^0_X$ but $(0,1;1,0) \in {\mathbb F}_1$ is not a singularity of $\mathcal{F}^1_X$.

Assume now that $\delta>1$. The point $(0,1;0,1) \in {\mathbb F}_\delta$ is the unique ordinary singularity of ${\mathcal F}_X^\delta$ belonging to the curve with equation $X_0=0$.  Let us see that it is a dicritical singularity. Indeed, the restriction of ${\mathcal F}_X^\delta$ to the open set $U_{11}$ of $\mathbb{F}_{\delta}$ determines a vector field which is given by the differential $1$-form:
\[
\omega^\delta:=(-x y^3+(\delta-1)y^3+\delta x^2 y^3+\delta x^{2 \delta+1}y)dx - (x^3y^2+x y^2+x^{2 \delta+2})dy.
\]
The origin is a singularity of $\omega^\delta$. To reduce this singularity we have to blow up $\omega^\delta$ and its strict transforms using changes of local coordinates of the type $(x=x', y=x'y')$;  the strict transform of $\omega^\delta$ after $n\leq \delta-2$ blowups is
\begin{multline*}
\tilde{\omega}^{\delta}(n):=(-x y^3 +(\delta-n-1) y^3+(\delta-n) x^2 y^3+ (\delta-n)x^{2 (\delta-n)+1} y )dx\\
-(x^3 y^2+ x y^2 + x^{2 (\delta-n)+2})dy.
\end{multline*}
In particular, $$
xa^{\delta}_3(n)+yb^{\delta}_3(n)=(\delta-n-2)x y^3,$$
where $ \tilde{\omega}_3^{\delta}(n)=
a^{\delta}_3(n)dx  +b^{\delta}_3(n)dy=(\delta-n-1)y^3dx-xy^2dy$
is the first nonvanishing jet of $\tilde{\omega}^{\delta}(n)$. Therefore, after $n=\delta-2$ blowups the origin becomes terminal dicritical and thus $(0,1;0,1)$ is a dicritical singularity of ${\mathcal F}_X^\delta$.

Thus, we have just proved that the conditions given in the statement of Theorem \ref{gordo} hold for $\delta_1=1$. However $X$ is not algebraically integrable, as we are going to prove. Indeed, consider the complex projective plane $\mathbb{P}^2$ with projective coordinates $(\mathcal{X}:\mathcal{Y}:\mathcal{Z})$ and the foliation ${\mathcal F}$ on $\mathbb{P}^2$  defined by the homogeneous $1$-form
$$
\Omega=(-\mathcal{X}^3 \mathcal{Z} - \mathcal{X}^2 \mathcal{Y} \mathcal{Z} -2 \mathcal{X} \mathcal{Y}^2 \mathcal{Z} -\mathcal{Y} \mathcal{Z}^3) d\mathcal{X} + (\mathcal{X}^3 \mathcal{Z} + \mathcal{X}^2\mathcal{Y}\mathcal{Z})d\mathcal{Y} + (\mathcal{X}^4 + \mathcal{X}^2 \mathcal{Y}^2 + \mathcal{X}\mathcal{Y} \mathcal{Z}^2) d\mathcal{Z}
$$
whose restriction to the open subset of $\mathbb{P}^2$ defined by $\mathcal{X}\neq 0$ (identified with $\mathbb{C}^2$) gives rise to the initial vector field $X$. Assume that $X$, and therefore ${\mathcal F}$, has a rational first integral $f$ and let us see that we get a contradiction. With notation as in Subsection \ref{sfh}, consider the subset ${\mathcal D}_{\mathcal F}\subseteq  {\mathcal C}_{\mathcal F}$ given by the (infinitely near) dicritical singularities of ${\mathcal F}$ and let
 $\pi_{\mathcal F} : Z_{\mathcal F} \rightarrow \mathbb{P}^2 $ be the birational map obtained by blowing up the points in ${\mathcal D}_{\mathcal F}$. ${\mathcal D}_{\mathcal{F}}$ consists of $6$ points, $p_1,\ldots,p_6$, such that $p_1\in \mathbb{P}^2$, $p_2$ and $p_6$ belong to the first infinitesimal neighbourhood of $p_1$ and, for $i\in \{3,4,5\}$, $p_i$ is a free point of the first infinitesimal neighbourhood of $p_{i-1}$.
 
Set $D_{\tilde{f}}$ a general fiber of $\tilde{f}:= f \circ \pi_{\mathcal F}:Z_{\mathcal F}\rightarrow \mathbb{P}^1$ and $\tilde{\mathcal F}$ the strict transform of ${\mathcal F}$ by $\pi_{\mathcal F}$. Since the unique terminal dicritical singularities in ${\mathcal D}_{\mathcal F}$ are $p_5$ and $p_6$, the strict transforms on $Z_{\mathcal F}$ of the exceptional divisors $E_{p_i}$, $i\in \{2,3,4\}$, (respectively, $E_{p_5}$ and $E_{p_6}$) are (respectively, are not) invariant by $\tilde{\mathcal{F}}$ (see Theorem \ref{jmaa}(2)); then, by \cite[Proposition 1(b)]{GalMon2014}, the divisor $D_{\tilde{f}}$ is linearly equivalent to $dL^*- (\alpha+\beta)E_{p_1}^*-\alpha \sum_{i=2}^5 E_{p_i}^*-\beta E_{p_6}^*$ for some positive integers $d,\alpha,\beta$, where $L^*$ (respectively, $E_{p_i}^*$) denotes the pull-back on $\zf$ of a general line of $\mathbb{P}^2$ (respectively, the exceptional divisor $E_{p_i}$).
 
The lines $C_1$ and $C_2$ with respective equations $\mathcal{X}=0$ and $\mathcal{Z}=0$ are invariant curves of the foliation ${\mathcal F}$. Both lines pass through the point $p_1$, the strict transform of $C_1$ (respectively, $C_2$) on the surface obtained by blowing up at $p_1$ passes through $p_2$ (respectively, $p_6$) and $p_3$ does not belong to the strict transform of $C_1$. This means that the  strict transform $\tilde{C}_1$ (respectively, $\tilde{C}_2$) of $C_1$ (respectively, $C_2$) on $Z_{\mathcal F}$ is linearly equivalent to the divisor $L^*-E_{p_1}^*-E_{p_2}^*$ (respectively, $L^*-E_{p_6}^*$). Again by \cite[Proposition 1(b)]{GalMon2014} one has that $d-(\alpha+\beta)-\alpha=0$ and $d-(\alpha+\beta)-\beta=0$ and, therefore, $D_{\tilde{f}}$ is linearly equivalent to the divisor $dL^*- (2d/3)E_{p_1}^*-(d/3) \sum_{i=2}^6 E_{p_i}^*$. A canonical divisor of the surface $Z_{\mathcal F}$ is $K_{Z_{\mathcal F}}=-3L^*+\sum_{i=1}^6 E_{p_i}^*$. Applying Formula (3) in 
\cite{GalMon2014}, we deduce that the canonical sheaf of the foliation $\tilde{\mathcal F}$ (see \cite[Section 2]{GalMon2014}) is ${\mathcal O}_{Z_{\mathcal F}}(K_{\tilde{\mathcal F}})$, where $K_{\tilde{\mathcal F}}=2L^*-E_{p_1}^*-E_{p_2}^*-E_{p_5}^*-E_{p_6}^*$. Thus, $(K_{Z_{\mathcal F}}-K_{\tilde{\mathcal F}})\cdot D_{\tilde{f}}=-d\neq 0$, which is a contradiction with \cite[Proposition 1(b) and Equality (2)]{GalMon2014}.
\end{example}

\section{The Newton polytope of the generic algebraic invariant curve}
\label{La5}

Given a polynomial $f(x,y)=\sum a_{ij} x^i y^j\in \C[x,y]$, the \emph{Newton polytope} of $f$, denoted by $\mathrm{Newt}(f)$, is the convex hull of the set $\{(i,j)\mid a_{ij}\neq 0\} \subseteq \mathbb{R}^2$. 

Let $X$ be an algebraically integrable complex planar polynomial vector field. Then, the Newton polytope $\mathrm{Newt}(g)$ of the generic algebraic invariant curve $g(x,y)$, associated to a primitive rational first integral $f$ of $X$, does not depend on the choice of $f$.

\begin{definition}
The \emph{Newton polytope} $\mathrm{Newt}(X)$ of an algebraically integrable complex planar polynomial vector field $X$ is defined as $\mathrm{Newt}(g)$, where $g(x,y)$ is the generic algebraic invariant curve associated to any primitive rational first integral of $X$.
\end{definition}

The following result studies the Newton polytope of a vector field as above.

\begin{theorem}\label{gordo2}
Let $X=a(x,y)\frac{\partial}{\partial x}+ b(x,y)\frac{\partial}{\partial y}$ be an algebraically integrable complex planar polynomial vector field such that $X\neq c \frac{\partial}{\partial y}$ and $X\neq c \frac{\partial}{\partial x}$ for all $c\in \mathbb{C}\setminus \{0\}$.  Consider the vector field $X'$ obtained from $X$ by swapping the variables $x$ and $y$, that is, $$X'=b(y,x)\frac{\partial}{\partial x}+ a(y,x)\frac{\partial}{\partial y}.$$
Let $\delta_1$ (respectively, $\delta'_1$) be the nonnegative integer introduced  in Theorem \ref{gordo} for the vector field $X$ (respectively, $X'$). Then, with notation as in the proof of Theorem \ref{gordo}, $\mathrm{Newt}(X)$ is contained in the following region:
$$\left\{ (u,v)\in \mathbb{R}^2_{\geq 0}\mid u\leq d_x^0+\delta_1v\;\mbox{and}\; v\leq d_y^0+\delta'_1 u \right\},$$
where $\mathbb{R}^2_{\geq 0}$ denotes the set of points of $\mathbb{R}^2$ with nonnegative coordinates.

\end{theorem}

\begin{proof}

Let $f=\frac{f_1(x,y)}{f_2(x,y)}$ be a primitive rational first integral of $X$ and set $$
g(x,y):=\alpha f_1(x,y)+\beta f_2(x,y)=\sum_{ij} g_{ij} x^i y^j\in \mathbb{C}(\alpha,\beta)[x,y]
$$
the associated generic algebraic invariant curve of $X$ as expressed in (\ref{rrr}).

Keep notation as given in the proof of Theorem \ref{gordo}. If the set $\Gamma$ defined in (\ref{gamma}) is empty,  then $\delta_1=0$ and $i\leq d_x^0$ for any nonzero coefficient $g_{ij}$ of the generic invariant curve (see the proof of Theorem \ref{gordo}); therefore the inequality $i\leq d_x^0+\delta_1 j$ holds trivially.

Assume now that $\Gamma$ is not empty and  let $k$ be the maximum of $\Gamma$ (notice that $\delta_1=\lceil k \rceil$). Pick $g_{ij}\in \mathrm{Coeff}(g)$.  If $j>0$ and $\frac{i-d_x^0}{j}\geq 0$ then $\frac{i-d_x^0}{j}\in \Gamma$ and therefore
$$
i\leq k j+d_x^0\leq d_x^0+\delta_1j.
$$
If $j>0$ and $\frac{i-d_x^0}{j}< 0$ then $i<d_x^0\leq d_x^0+\delta_1 j$.  Finally, if $j=0$, $i\leq d_x^0$ by the definition of $d_x^0$.

Reasoning analogously with the vector field $X'$ and since it is algebraically integrable with generic algebraic curve $g(y,x)$, it holds that $j\leq d_y^0+\delta'_1 i$ for all $(i,j)$ such that $g_{ij}\in \mathrm{Coeff}(g)$. This concludes the proof.

\end{proof}

As a consequence of Theorem \ref{gordo2}, we provide, under certain assumptions, a bound on the degree of a primitive rational first integral of an algebraically integrable planar vector field that depends only on the values $\delta_1$, $\delta'_1$, $d_x^0$ and $d_y^0$.

\begin{corollary}
\label{El52}
With assumptions and notation as given in Theorem \ref{gordo2}, suppose that $\delta_1=0$  (respectively, $\delta'_1=0$). Then the degree of a primitive rational first integral of $X$ is bounded from above by $(1+\delta'_1)d_x^0+d_y^0$ (respectively, $(1+{\delta}_1)d_y^0+d_x^0$).
\end{corollary}

\begin{remark}
Let $X$ be an algebraically integrable planar vector field. Then the value $d_x^0$ (respectively, $d_y^0$) coincides with the total intersection number between a general integral algebraic invariant curve of $X$ and the line $y=0$ (respectively, $x=0$).
\end{remark}

We conclude this paper with a result about  complex planar vector fields $X$ having a rational first integral of a specific type. Firstly, notice that $X$ has a primitive rational first integral of the form
\begin{equation}\label{form}
\frac{a+xyH_1(x,y)}{b+xyH_2(x,y)},
\end{equation}
with $H_1,H_2\in \mathbb{C}[x,y]$ and $(a,b) \in \mathbb{C}^2 \setminus \{(0,0)\}$ if and only if $d_x^0=d_y^0=0$.
\begin{corollary}\label{ccc3}
Let $X$ be a complex planar vector field and keep notation as given in Theorem \ref{gordo2}.
\begin{itemize}
\item[(a)] If $X$ has a primitive rational first integral of type (\ref{form}), then the Newton polytope of $X$, $\mathrm{Newt}(X)$, is contained in the convex cone $$\Psi_X:=\left\{ (u,v)\in \mathbb{R}^2_{\geq 0}\mid u\leq \delta_1 v\;\mbox{and}\; v\leq \delta'_1 u \right\},$$
which can be computed only from $X$.

\item[(b)] If $\delta_1=0$ or $\delta'_1=0$, then $X$ has no primitive rational first integral of type (\ref{form}).

\end{itemize}
\end{corollary}

\begin{proof}

Part $(a)$ is straightforward from Theorem \ref{gordo2}. Part $(b)$ follows because, if $X$ had a rational first integral of the form (\ref{form}) and either $\delta_1=0$ or $\delta'_1=0$, then, by Part $(a)$, the set $\Psi_X$ would be $\{(0,0)\}$, which is a contradiction.

\end{proof}

\bibliographystyle{plain}
\bibliography{MIBIBLIO}

\end{document}